\documentclass[a4paper,12pt]{amsart}

\usepackage[all]{xy}
\usepackage{graphics}
\usepackage{color}
\usepackage[T1]{fontenc}
\usepackage{hyperref}

\usepackage{parskip}
\usepackage[upright]{fourier}
\usepackage{eucal}
\newcommand{\B}[1]{{\mathbf #1}}
\newcommand{\C}[1]{{\mathcal #1}}
\newcommand{\F}[1]{{\mathfrak #1}}

\newtheorem{theorem}[equation]{Theorem}
\newtheorem{corollary}[equation]{Corollary}
\newtheorem{lemma}[equation]{Lemma}
\newtheorem{proposition}[equation]{Proposition}
\theoremstyle{definition}
\newtheorem{example}[equation]{Example}
\theoremstyle{remark}

\newtheorem{remark}[equation]{Remark}
\newtheorem{question}[equation]{Question}
\numberwithin{equation}{section}
\numberwithin{figure}{section}
\numberwithin{table}{section}

\newcommand{\OP}{\operatorname}
\newcommand\Symp{\OP{Symp}}
\newcommand\Ham{\OP{Ham}}
\newcommand\Isom{\OP{Isom}}
\newcommand\Diff{\OP{Diff}}
\newcommand{\la}{{\lambda}}

\begin{document}

\title[Two-cocycle]{A two-cocycle on the group of symplectic diffeomorphisms}
\author{\'Swiatos\l aw Gal}
\address{
\'S.G.:
Uniwersytet Wroc{\l}awski \& Universit\"at Wien\\
}
\email{
\tt sgal@math.uni.wroc.pl
}
\author{Jarek K\k edra}
\address{
J.K.:
University of Aberdeen \& Uniwersytet Szczeci\'nski\\
}
\email{
\tt kedra@abdn.ac.uk
}

\thanks{The first author is was partially supported by Polish MNiSW grant N N201 541738} 
\begin{abstract}
We investigate the properties of a two-cocycle on the group of
symplectic diffeomorphisms of an exact symplectic manifold defined by
Ismagilov, Losik, and Michor. We
provide both vanishing and nonvanishing results and applications to
foliated symplectic bundles and to Hamiltonian actions of finitely
generated groups.
\end{abstract}

\maketitle


\section{Introduction}
Let $(M,d\la)$ be a connected  exact symplectic manifold with trivial first
real cohomology, $H^1(M;\B R)=0$. In this paper we investigate a two-cocycle
$\F G$ on the group $\Symp(M,d\la)$ of symplectic diffeomorphisms of
$(M,d\la)$.  This cocycle was defined by R.~Ismagilov, M.~Losik, and P.~Michor
in~\cite{MR2270616} where they proved that it is cohomologically nontrivial
when $M$ is either the standard symplectic $\B R^{2n}$ or a Hermitian symmetric
space.  The following theorem generalises their results.

\begin{theorem}\label{T:sas}
Let $(M,d\la)$ be the universal cover of a {\em closed}
symplectic manifold $(X,\sigma)$. The cocycle
$\mathfrak G$ represents a nonzero cohomology class.
\end{theorem}

A symplectic manifold $(X,\sigma)$ whose universal cover is
exact is called {\bf symplectically aspherical}. The reason
is that the property can be equivalently characterised by
the vanishing of the symplectic area of every sphe\-re in $X$.
More precisely, $(X,\sigma)$ is symplectically aspherical if and
only if
$$
\int_{S^2}s^*\sigma =0
$$
for every smooth map $s\colon S^2\to X$  
(see \cite{MR2402905} for a survey).

\subsection{Vanishing properties}
It is interesting to ask about the restriction of the cocycle $\F G$ to
various subgroups of the group symplectic diffeomorphisms of
$(M,d\la)$. It turns out that its cohomology class vanishes on the
subgroup of compactly supported symplectic diffeomorphisms $\Symp_c(M,d\la)$
and on sub\-groups preserving certain isotropic submanifolds.

\begin{theorem}\label{T:vanish-cpt}
The cocycle $\F G$ restricted to $\Symp_c(M,d\la)$ represents the
trivial cohomology class.
\end{theorem}

A subset $L\subset M$ is {\bf isotropic} if $i^*\la$, where $i\colon
L\to~M$ is the inclusion map, is a closed one-form. We say that
isotropic submanifold $L\subset M$ is {\bf exact isotropic}
if $i^*\la$ is exact.  This is to say that $[i^*\la] = 0$ in
$H^1(L;\B R)$.  It is always the case if $b_1(L)=0$.

Let $\Symp_L(M,d\la):=\{f\in \Symp(M,d\la)\,|\,f(L)=L\}$ 
be the group of symplectic diffeomorphisms
preserving the submanifold $L$.

\begin{theorem}\label{T:vanish-isotr}
Let $i\colon L\to M$ be the inclusion of a closed connected exact isotropic
submanifold.  Then 
$\F G$ restricted to the group $\Symp_L(M,d\la)$ represents the 
trivial cohomology class.
\end{theorem}

If $(M,d\la)$ is the universal cover of a closed symplectic manifold
$(X,\sigma)$ then the group $\Ham(X,\sigma)$ can be viewed as a
subgroup of $\Symp(M,d\la)$ (see page \pageref{SS:ham} for details). 

\begin{theorem}\label{T:ham}
The cocycle $\F G$ restricted to $\Ham(X.\sigma)$
represents the trivial cohomology class.
\end{theorem}

\subsection{Hermitian symmetric spaces}\label{SS:hermitian}
If $M$ is a Hermitian symmetric space of noncompact type
(see~\cite{MR510552} or \cite[Chapter 3]{MR2205508} for definitions) 
then the connected component
$G:=\Isom^\circ(M)$ of the group of the isometries of the K\"ahler metric
admits a nontrivial bounded two-cocycle $\F K$ called the 
{\em K\"ahler cocycle}. It is defined by the integration of the K\"ahler form over
geodesic trilaterals.  More precisely, fix a reference point $x\in M$ and
define $$ \F K(g,h) := \int_{\triangle}\sigma, $$ where $\triangle\subset M$ is
a geodesic trilateral with vertices $x,g(x),gh(x)$ and $\sigma \in \Omega^2(M)$
is the K\"ahler form.

It is known \cite[Section 5.2]{MR2205508} that the K\"ahler cocycle is
bounded.  If $\Gamma \subset G$ is a uniform lattice (i.e. a discrete
subgroup such that the quotient $\Gamma\backslash G$
is compact) then the pull
back of the K\"ahler cocycle represents a nontrivial cohomology class
in $H^2(\Gamma;\B R)$. This class is equal to the class represented by
the K\"ahler form of the compact orbifold $X= \Gamma\backslash M$. A
detailed presentation can be found in Wienhard~\cite[Chapter 5]{MR2205508}.

The following result was proved in \cite[Theorem 5.1]{MR2270616}.

\begin{theorem}\label{T:hermitian}
Let $(M,d\la)$ be a Hermitian symmetric space of noncompact type and
let $G\subset \Symp(M,d\la)$ be the connected component of the group
of isometries of the K\"ahler metric.  Then the pullback of the
cocycle $\F G$ to $G$ is cohomologous to the
K\"ahler cocycle $\F K$.
\end{theorem}

In Proposition \ref{P:primitive}, we observe that, under suitable
choices made, the cocycles $\F G$ and $\F K$ are in fact equal.

\subsection{Boundedness properties}
Theorem \ref{T:hermitian} shows that the restriction of $\F G$ to
a certain subgroup is a bounded two-cocycle. 
This is not the case in general.

\begin{theorem}\label{T:unbounded}
The two-cocycle $\F G$ is unbounded on $\Symp(M,d\la)$. Moreover,
if $(M,d\la)$ is the universal cover of a closed symplectic
manifold $(X,\sigma)$ then the restriction of $\F G$ to
$\Ham(X,\sigma)$ is unbounded.
\end{theorem}

We investigate boundedness properties of $\F G$ in Section
\ref{S:boundedness}. We then apply these properties to prove
a theorem of Polterovich about Hamiltonian actions of
finitely generated groups on symplectically hyperbolic manifolds.
This and other applications
are presented in Section \ref{S:applications}.

\subsection*{Acknowledgements}
We warmly thank Dusa McDuff for explaining us the proof of
Theorem \ref{T:ham} and Example \ref{E:foliated} (4).
We thank Dieter Kotschick for drawing our attention to the
paper of Ismagilov, Losik and Michor.  And, last but not least,
anonymous referee for helpful remarks improving the final exposition.

\section{Definition of the cocycle}\label{S:definition}
We refer the reader to \cite{MR83k:20002}
for the standard definitions and facts about cohomology of groups.
In particular, we would make use of the fact that group
cohomology is the same as the cohomology of the classifying space of the group
when the group is endowed with discrete topology \cite[Section I.4]{MR83k:20002}.
A topological group $G$ considered with the discrete topology will 
be denoted $G^d$.


Let $(M,d\la)$ be an exact symplectic manifold with 
$H^1(M;\B R)=~0$. If $g\colon M\to~M$ is a symplectic diffeomorphism
then the one-form $g^*(\la) - \la$ is closed. Thus
the integral $\int_{\ell}g^*(\la) - \la $ 
depends only on the endpoints  of the path 
$\ell\colon [0,1]\to M$. In what follows we shall
denote this integral by
$$
\int_x^yg^*(\la) - \la,
$$
where $x=\ell(0)$ and $y=\ell(1)$.

Let $x\in M$ be a reference point. Following Ismagilov, Losik and
Michor \cite{MR2270616}, we define a two-cocycle $\F G_{x,\la}$ on
the discrete group $\Symp(M,d\la)$ of symplectic diffeomorphisms of
$(M,d\la)$ by
$$
\F G_{x,\la}(g,h):=\int_{x}^{h(x)}g^*\la - \la. 
$$ 
We shall omit the subscripts when it does not lead to a
confusion.

The proof of the following proposition is straightforward,
cf. \cite[Theorem 3.1]{MR2270616}.

\begin{proposition}\label{P:definition}
The map $\F G$ satisfies each of the following conditions:
\begin{enumerate}
\item 
$\F G$ is a two-cocycle on $\Symp(M,d\la)$,
\item
the cohomology class 
$[\F G]\in H^2(\Symp(M,d\la);\B R)$
does not depend on the choice of $x\in M$,
\item
the cohomology class $[\F G]$ does not depend of the choice of primitive~$\la$.
\end{enumerate}\qed
\end{proposition}

\subsection{An alternative approach via a one-cocycle}

Let $g\in \Symp(M,d\la)$.  Recall that $g^*\la - \la$ is exact since
$g$ preserves $d\la$ and $b_1(M)$ is assumed to vanish.
Therefore there exists a function $\C K_{\la}(g)$
unique up to a constant (as we have assumed that $M$ is connected)
such that
$$
d\C K_{\la}(g) = g^*\la - \la.
$$
The map
$$
\C K_{\la}\colon \Symp(M,d\la)\to C^{\infty}(M)/\B R
$$
is a one-cocycle on the group of symplectic diffeomorphisms
of $(M,d\la)$ with values in the right representation of smooth
functions on $M$ modulo the constants. The action of a
diffeomorphism on a function is by the composition.
That is, the map $\C K_{\la}$ satisfies the following identity
\begin{equation}\label{eq:cocycle}
\C K_{\la}(gh) = \C K_{\la}(g)\circ h + \C K_{\la}(h),
\end{equation}
which is straightforward to check. This cocycle has been
investigated by the authors in \cite{GK1}.

Consider the following short exact sequence of 
$\Symp(M,d\la)$-rep\-re\-sen\-ta\-tions 
$$
0\to\B R\to C^{\infty}(M) \to C^{\infty}(M)/\B R\to0.
$$ 

\begin{proposition}\label{P:bock}
Consider the connecting homomorphism 
$$
\delta \colon H^1(\Symp(M,d\la),C^{\infty}(M)/\B R)\to
H^2(\Symp(M,d\la);\B R).
$$
corresponding to the above extension of representations
(see Brown \cite[III.6]{MR83k:20002} for definition).
Then $\delta [\C K_{\la}]=[\F G]$.
\end{proposition}

\begin{proof}
We start with choosing a lift
$\widetilde {\C K}_{\la}\colon\Symp(M,d\la)\to C^\infty(M)$
of  $\C K_{\la}$.
This may be obtained by setting $\widetilde{\C K}_{\la}(g)(x)=~0$.
It follows from the identity (\ref{eq:cocycle}) that  the coboundary
$$
\delta \widetilde {\C K}_{\la}(g,h)\colon=
\widetilde {\C K}_{\la}(g)\circ h-\widetilde {\C K}_{\la}(gh)+\widetilde {\C K}_{\la}(h)
$$
belongs to $\B R$, i.e. it is a constant function.
Therefore, without loss of generality, we can evaluate it at point~$x$:
\begin{eqnarray*}
\delta \widetilde  {\C K}_{\la}(g,h)(x)&=&
\widetilde {\C K}_{\la}(g)(hx)-\widetilde {\C K}_{\la}(gh)(x)+\widetilde {\C K}_{\la}(h)(x)\\
&=&\widetilde {\C K}_{\la}(g)(hx)\\
&=&\widetilde {\C K}_{\la}(g)(hx) - \widetilde {\C K}_{\la}(g)(x)\\
&=& \int_x^{h(x)} d{\C K}_{\la}\\
&=& \int_x^{h(x)} g^*\la - \la \\
&=&\F G_{x,\la}(g,h).\\
\end{eqnarray*}
\end{proof}

\begin{remark}\label{R:one-two} 
Notice that
$\F G_{x,\la}(g,h) = \C K_{\la}(g)(h(x)) - \C K_{\la}(g)(x).$
\end{remark}

\section{Proofs of the results}\label{S:proofs}

Let us choose a path $\ell_{x,y}\colon [0,1]\to M$ from the
basepoint $x\in M$ to a point $y\in M$. Let
$-\ell_{x,y}(t):=\ell_{x,y}(1-t)$.
Let $\F k\in C^1(\Symp(M,d\la);\B R)$ be a cochain defined by 
$$\F k(g):= \int_{\ell_{x,g(x)}}\la.$$

\begin{lemma}\label{L:trilateral}
Let $\triangle\subset M$ be a trilateral with sides
$\ell_{x,g(x)}$, $g\ell_{x,hx}$, $-\ell_{x,ghx}$. Then
$$
(\F G + \delta \F k)(g,h) = \int_{\triangle}d\la.
$$
\end{lemma}

\begin{proof}
It is the following direct calculation.
\begin{eqnarray*}
(\F G + \delta \F k)(g,h) 
&=&\left(\int_{g\ell_{x,hx}}\la - \int_{\ell_{x,hx}} \la\right)
+\left( \int_{\ell _{x,gx}}\la - \int_{\ell_{x,ghx}}\la + \int_{\ell_{x,hx}}\la\right)\\
&=&\int_{g\ell_{x,hx}}\la + \int_{\ell_{x,gx}}\la - \int_{\ell_{x,ghx}}\la \\
&=&\int_{\triangle}d\la. 
\end{eqnarray*}
\end{proof}

\begin{remark}\label{R:paths}
Notice that the choice of paths cannot be made continuous if $M$
is not contractible. That is, the path fibration $PM\to M\times M$
defined by $\ell \mapsto (\ell(0),\ell(1))$ does not admit a
continuous section in general.
\end{remark}

\subsection{Proof of Theorem \ref{T:sas}.}  
Let $M\to X$ be the universal cover. Consider the composition 
$X \to B\pi_1(X)\to B\Symp(M,d\la)^d$ of the map classifying the
universal cover followed by the map induced by the inclusion
$\pi_1(X)\subset \Symp(M,d\la)$ as the deck transformations. 
The strategy is
to show that pullback of the cocycle $\F G$ with respect to this map
represents the class $[\sigma]$ of the symplectic form. In fact, we
shall prove the following more general result.

\begin{theorem}\label{T:better-sas}
Let $(M,d\la)$ be a connected regular cover of a closed 
symplectic $2n$-manifold
$(X,\sigma)$. Let $\Gamma \subset \Symp(M,d\la)$ denote the deck
transformation group.  Suppose that $H^1(M;\B R)=0$. Then the
pullback of the class $[\F G]$ by the homomorphism induced by the
composition $X \to B\Gamma \to B\Symp(M,d\la)^d$ is equal to the class
represented by the symplectic form~$\sigma$. In particular, 
$[\F G]^n \neq 0$.
\end{theorem}

\begin{proof}
The classifying space $B \Gamma$ is constructed as a 
realisation of a simplicial set according to Milnor
\cite{MR0077122,MR0077932}. 
In this simplicial complex the set of $n$-simplices is
identified with $G^n$. This gives an identification of
the cochain complex for the group cohomology and
the simplicial cochain complex
(see \cite[Section II.1.1.B]{MR2001a:22001} for details).

Let us choose a {\sc cw}-complex structure with a 
single vertex in $X$. It is always possible due 
to a standard argument (see for example 
\cite[Proposition 4.2.13]{MR2000h:57038}
and the subsequent discussion). Such a structure
induces a {\sc cw}-structure on the covering $M$.
Notice that the vertices (zero-cells) of this
induced structure can be identified with $\Gamma$.
Let $x\in M$ be a reference vertex.

With the above choice the classifying map 
$c\colon X\to B\Gamma$ is cellular on the one-skeleton
and after an appropriate subdivision of $X$ it can be
made cellular on the two-skeleton. Here we consider
the simplicial structure on $B \Gamma$ as
a {\sc cw}-complex.

Let $\Delta$ be an oriented two-cell of $X$. Its image
$c(\Delta)$ with respect to the classifying map
is a two-simplex in $B \Gamma$ and hence a pair
of elements from $\Gamma$. To find these elements
consider the lift of $\Delta$ to $M$ passing
through the reference vertex $x$. Let $\Delta_x$ 
denote this lift. The vertices of $\Delta_x$
are of the form $x,gx,ghx$ for some $g,h\in \Gamma$.
Thus $c(\Delta)$ is identified with the pair
$(g,h)\in \Gamma\times\Gamma$.

Let us represent the cohomology class $[\F G]$ by the
cocycle $\F G+\delta \F k$ as in Lemma \ref{L:trilateral}.
Then we pull it back to $\Gamma$ and consider it as
a {\sc cw}-cocycle, pull it back to $X$ and evaluate
on a two-cell $\Delta$.
\begin{eqnarray*}
\langle c^*(\F G+\delta \F k),\Delta\rangle 
&=& \langle \F G+\delta \F k,c(\Delta)\rangle\\
&=&(\F G+\delta \F k)(g,h)\\
&=& \int_{\Delta_x}d\la\\
&=& \int_{\Delta}\sigma
\end{eqnarray*}
That is, the pull back of the cocycle $\F G +\delta\F k$ to $X$ is
a cocycle defined by the integration of the symplectic
form $\sigma$. Thus it represents the {\sc cw}-cohomology 
class corresponding to the cohomology class
of $\sigma$ under the de Rham isomorphism.
In particular, since $[\sigma]^n\neq 0$, we get that
$[\F G]^n\neq 0$.
\end{proof}

\begin{question}
Are the higher powers of the cohomology class 
$[\F G]$ nonzero?
\end{question}

\subsection{Proof of Theorem \ref{T:hermitian} 
{(cf. \cite[Section 4.2]{MR2270616})}}
Recall that we need to show that the pullback of the cocycle $\F G$
with respect to the inclusion $G=\Isom^\circ(M)\subset \Symp(M,d\la)$
is cohomologous to the K\"ahler cocycle~$\F K$.
It immediately follows from Lemma \ref{L:trilateral}.
\qed

As we pointed out in the introduction a stronger statement is true.

\begin{proposition}\label{P:primitive}
Let $x\in M$ be a reference point.
There exists a primitive $\la$ such that
the cocycle $\F G_{x,\la}$ is equal to the K\"ahler cocycle $\F K$.
\qed
\end{proposition}

\begin{proof}
A primitive $\la$ can be written as $\la=-Jd\varphi$ where $J$
is the complex structure on $M$.  The function $\varphi$ is called the
{\em K\"ahler potential}.  Averaging $\varphi$ with respect to the
(compact) stabiliser of a point $x$, one can choose $\varphi$ to be
radial (as the stabiliser of $x$ in $G$ acts transitively on the unit
tangent sphere at $x$), i.e. $\varphi=u(\OP{dist}(\cdot,x))$ for a
suitable function $u\colon [0,\infty)\to \B R$.

Let $L$ be the Liouville vector field defined by $i_Ld\la=\la$.
By definition of~$\la$, the vector field $L$ is the metric
gradient of the function $\varphi$.  Recall that the spheres around
$x$ (the level sets of the function $\varphi$) are orthogonal to the
geodesics from $x$.  Therefore the flow  of $-L$ contracts $M$ to
the unique zero $x$ of $L$ along the geodesics.  Let $\ell_{x,y}$ denote
the geodesic between $x$ and $y$.

A flow line of the Liouville vector field $L$ is $\la$-null, as
$i_L\la=(i_L)^2d\la=0$.  Thus $\int_{\ell_{x,y}}\la=0$,  
for every $y$.  Therefore
$$
\F K(g,h)=
\int_{\ell_{x,gx}}\la
+\int_{\ell_{gx,ghx}}\la
-\int_{\ell_{x,ghx}}\la
=\int_{\ell_{gx,ghx}}\la,
$$
and
$$
\F G_{x,\la}(g,h)=
\int_{\ell_{gx,ghx}}\la
-\int_{\ell_{x,hx}}\la
=\int_{\ell_{gx,ghx}}\la.
$$
\end{proof}

\begin{example}
If $M=\OP{U}(1,1)/\OP{SO}(2)$ is a complex hyperbolic line then
the function $u\colon [0,\infty)\to \B R$ 
from the first paragraph of the above proof
is defined by $u(r)=\log(\cosh(r)+1)$.
\hfill $\diamondsuit$
\end{example}

\begin{question}
What are the maximal subgroups of $\Symp(M,d\la)$ on which the
cocycle $\F G$ is cohomologous to a bounded one?
\end{question}

\subsection{Proof of Theorem \ref{T:vanish-cpt}.}

Observe that $g^*\la-\la$ vanishes outside the support of $g$.
Therefore if $g$ has a compact support, one may try to normalise $\C K(g)$
to vanish outside the support of $g$ as well.  However in \cite[p.~78]{GK1}
we construct an example (with $M=T^\vee S^1$, the cotangent bundle of $S^1$) where $\C K(g)$ takes different
values on both ends of $M$.  Nevertheless, one can fix an end of $M$ and declare
$\C K(g)$ to vanish there.  This provides a lift of $\C K$ to functions on $M$
(without constant ambiguity).  Thus the connecting
homomorphism sends $\C K$ to zero (cf.~Proposition \ref{P:bock}).
\qed

\begin{remark}
We have the following alternative argument. Since $M$, being a manifold, 
is $\sigma$-compact, there exists a ray
$\gamma\colon [0,\infty)\to~M$ starting at $x$ and
leaving any compact subset of $M$.  For $g\in\Symp_c(M,d\la)$
define $\F b(g):=\int_\gamma g^*\la-\la$.  Notice that this
makes sense as eventually, along $\gamma$, outside the support of $g$
one has $g^*\la=\la$. We have the following computation
in which a curve from $x$ to $h(x)$ is chosen to be the concatenation
of a part of $\gamma$ from $x$ to the outside of the union of the
supports of $g$ and $h$ and then the part of
$-h(\gamma)$ back to $h(x)$.

\begin{eqnarray*}
\F G_{x,\la}(g,h) &=& \int_x^{h(x)}g^*\la - \la\\
&=&\int_{\gamma} g^*\la - \la - \int_{h(\gamma)}g^*\la - \la\\
&=&\int_{\gamma} g^*\la - \la - \int_{\gamma}(gh)^*\la - \la
+ \int_{\gamma}h^*\la - \la= \delta \F b(g,h)
\end{eqnarray*}
\end{remark}

\subsection{Proof of Theorem \ref{T:vanish-isotr}.}
Recall that $L$ is an exact isotropic submanifold of $M$.
Assume that $i^*\la=0$ and choose $x\in L$.  Then 
$$ 
\mathfrak G_{x.\la}(g,h)=\int_x^{h(x)}g^*\la - \la=0, 
$$ 
since the curve joining $x$ and $h(x)$ can be chosen to be contained
in $L$ and $i^*(g^*\la)=g^*(i^*\la)=0$.

Observe that we can always find a primitive $\la$ such
that $i^*\la =~0$. Indeed, let $\la'$ be a primitive
with the property that $[i^*\la']=0$.
We have $i^*(\la') = dF'$ for some function 
$F'\colon L\to \B R$. Extending $F'$ to a function $F\colon M\to \B R$
we obtain $i^*(\la' - dF)=0$ and we take 
$\la:= \la' - dF$.

\qed

\begin{example}\label{E:vanish}
The cohomology class $[\F G]$ vanishes on the 
following subgroups of $\Symp(M,d\la)$:
\begin{enumerate}
\item
$\Symp(M,x)$ -- the isotropy of a point $x\in M;$
\item
$\Diff(L)\subset\Symp(T^\vee L)$ where $M=T^\vee L$ is the cotangent bundle of~$L$.
\end{enumerate}
\hfill $\diamondsuit$
\end{example}

\begin{example}
Let $(M,d\la)$ be the universal cover of $(X,\sigma)$.
The deck transformation group $\pi_1(X)\subset \Symp(M,d\la)$
preserves the orbit of $x\in M$. Such an orbit is clearly isotropic.
This shows that the connectivity of $L$ is essential for Theorem
\ref{T:vanish-isotr} to hold, according
to Theorem~\ref{T:sas}. 

Let $S\subset \pi_1(X)$ be a finite set of generators.
The associated Cayley graph $\Gamma_S$ can be embedded in $M$ 
as a connected isotropic subspace invariant under the 
deck transformations. To do this consider the map from a 
wedge of circles $Y$, one per generator
of $\pi_1(X)$ and map it into $X$.  
Then the Cayley graph $\Gamma_S$ is a covering
of $Y$ and the map lifts to the equivariant map into $M$.
The primitive $\la$ represents a nontrivial cohomology class
of $\Gamma_S$. This shows that the hypothesis that $[i^*\la]=0$
is also essential.

In this example the isotropic subspace is
not a submanifold but it can be improved by taking a surface of genus
equal to the number of generators of $G$ and mapping it as an isotropic
subset into $X$.  This can be done provided the dimension of $X$ is big enough.
The lift to $M$ is a  $\pi_1(X)$-invariant closed isotropic submanifold of $M$.
\hfill $\diamondsuit$
\end{example}

\subsection{Proof of Theorem \ref{T:ham}.}\label{SS:ham}
Let us explain first that if $(X,\sigma)$ is a closed symplectic
manifold with an exact universal cover $(M,d\la)$ then there is an
injective homomorphism $\Ham(X,\sigma)\to \Symp(M,d\la)$.

Let $f_t\in \Ham(X,\sigma)$ be an isotopy from the identity to
$f=f_1$. This isotopy can be lifted to an isotopy 
$\tilde f_t\in \Symp(M,d\la)$ from the identity to $\tilde f=\tilde f_1$.
Since the evaluation map $\Ham(X,\sigma)\to X$ induces the trivial
homomorphism on the fundamental group 
\cite[Corollary 9.1.2]{MR2045629}, the endpoint $\tilde f$ does not
depend on the choice of the isotopy $f_t$.

We shall prove that $\F G$ restricted to $\Ham(X,\sigma)$ is a
coboundary. Recall that, due to Proposition \ref{P:bock}, 
$[\F G]=\delta [\C K_{\la}]$. We shall show that
the restriction of the cocycle $\C K_{\la}$ to $\Ham(X,\sigma)$ admits
a lift to a cocycle 
$\widetilde {\C K}_{\la}\colon \Ham(X,\sigma)\to C^{\infty}(M)$.

Let $f_t\in \Ham(X,\sigma)$, for $t\in [0,1]$ be a Hamiltonian isotopy
from the identity to $f=f_1$ generated by a normalised Hamiltonian
function $H_t\colon X\to \B R$. Recall that $H_t$ is
normalised if $\int _X H_t\sigma^n = 0$ for all $t\in[0,1]$.  
Let $F_t\colon M\to \B R$ be defined by
\begin{equation}\label{Eq:F_t}
F_t(x)= \int_0^t (\la(X_s) + \widetilde H_s)(\tilde f_s(x))ds. 
\end{equation}
Here $\widetilde H_t$ is the lift of the Hamiltonian $H_t$ and
$X_t$ is the corresponding vector field. According to
\cite[Proposition 2.8]{GK1} (cf. \cite[Proposition 9.19]{MR2000g:53098};
beware that \cite{MR2000g:53098} uses the opposite sign convention
for Hamiltonians)
we have 
$$
dF_t = d\C K_{\la}(\tilde f_t).
$$

Let $\widetilde {\C K}_{\la}(f) := F_1$.  We need to check that this
definition does not depend on the choice of isotopy from the identity
to $f$. Let $\{f_t\}$ and $\{f'_t\}$ be two Hamiltonian 
isotopies from the identity to $f\in \Ham(X,\sigma)$.  
Let $\{\tilde f_t\}$ and $\{\tilde f'_t\}$ denote their lifts
to $\Ham(M,d\la)$. As explained above $\tilde f_1=\tilde f'_1$.
The formula \ref{Eq:F_t} defines two time-dependent
functions $F_t$ and $F'_t$. Observe that the difference
$F_1-F'_1$ is constant because
$$
d(F_1-F'_1) = d(\C K_{\la}(\tilde f_1)-\C K_{\la}(\tilde f'_1)).
$$
The following calculation shows that this constant is
equal to the function corresponding to the concatenation
of the isotopy $\{f_t\}$ and the isotopy $\{f'_{1-t}\}$.
Let $g\colon [0,2]\to \Ham(X,\sigma)$ denote this concatenation
and let $G$ and $Y$ denote its Hamiltonian function and the
generated vector field respectively.
\begin{eqnarray*}
\lefteqn{\int_0^2(\la(Y_s)+\widetilde G_s)(\tilde g_s(x))ds}\\
&=& \int_0^1(\la(Y_s)+\widetilde G_s)(\tilde g_s(x))ds+
\int_1^2(\la(Y_s)+\widetilde G_s)(\tilde g_s(x))ds \\  
&=& \int_0^1(\la(Y_s)+\widetilde G_s)(\tilde g_s(x))ds+
\int_0^1(\la(Y_{2-t})+\widetilde G_{2-t})(\tilde g_{2-t}(x))dt \\
&=& \int_0^1(\la(X_s)+\widetilde H_s)(\tilde f_s(x))ds -
\int_0^1(\la(X'_{t})+\widetilde H'_{t})(\tilde f'_{t}(x))dt \\
&=& F_1 - F'_1 
\end{eqnarray*}

Consequently, the proof is reduced to showing that if $\{f_t\}$ is a
loop in $\Ham(X,\sigma)$ based at the identity then $F_1(x) = 0$ for
all $x\in M$. We have that
\begin{equation}\label{Eq:action}
F_1(x) =
\int_{\tilde f_t(x)}\la + \int_0^1 \widetilde H_t(\tilde f_t(x))dt.
\end{equation}
and this quantity is known as the action functional of the Hamiltonian
loop $\{f_t\}$. According to Schwarz \cite[Lemma 3.3]{MR1755825},
$F_1$ is constant and depends only on the homotopy class of the
loop $\{f_t(x)\}$.  Finally, it follows from the proof of Proposition
3.1 (i) in McDuff \cite[page 311]{MR2563682} that $F_1(x)$ is equal to
zero. This finishes the proof of Theorem \ref{T:ham}.\qed

It is important that we ask about vanishing of the cocycle on the group
of Hamiltonian diffeomorphisms of a compact quotient of exact symplectic
manifold $M$, i.e. the group
generated by periodic (with respect to the action of the
deck-transformations group $\Gamma$) Hamiltonians.
We already know that the cocycle is nontrivial on $\Ham(M)$ when $M$
is a symmetric space of Hermitian type.
This motivates the following question.

\begin{question}
Does $\F G$ vanish on the group generated by bounded Hamiltonians on $M$?
Given a complete Riemannian metric on $M$, does $\F G$ vanish on the group
generated by Hamiltonians with bounded differential?
\end{question}

\section{Boundedness properties of $\F G$}\label{S:boundedness}

Let $\F c$ be a real valued two-cocycle on a group $G$.  An element
$g$ of $G$ defines a function $\,\vphantom{\F c}^g\F c\colon G\to \B R$
by the formula $$\,\vphantom{\F c}^g\F c(h)=\F c(g,h).$$
%
We say that $\F c$ is {\bf  semibounded} if $\vphantom{\F c}^g\F c$ is a bounded
function on $G$ for any $g\in G$.  By $|g|_\F c$ we denote the supremum norm
of $\,\vphantom{\F c}^g\F c$:
$$
|g|_\F c:=\sup_{h\in G}|\F c(g,h)|.
$$

\begin{lemma}\label{L:two}
Assume that $\F c$ is a semibounded two-cocycle on $G$.  Then for all $f,g\in G$
$$|fg|_\F c\leq 2|f|_\F c+|g|_\F c.$$
\end{lemma}

\begin{proof}
By the cocycle identity
\begin{eqnarray*}
|fg|_\F c&=&\sup_h|\F c(fg,h)|\cr
&\leq&\sup_h\left(|\F c(f,g)|+|\F c(f,gh)|+|\F c(g,h)|\right )\cr
&\leq&2\sup_h|\F c(f,h)|+\sup_h|\F c(g,h)|.\cr
&=&2|f|_\F c+|g|_\F c.\cr
\end{eqnarray*}
\end{proof}

A closed symplectic manifold $(X,\sigma)$ is called 
{\bf symplectically hyperbolic} if the pullback of the symplectic
form $\sigma$ to the universal cover is exact and admits a primitive
that is bounded with respect to the Riemannian metric induced from an
auxiliary metric on $X$ \cite[Definition 1.2.C]{MR2003i:53126}.
Examples and constructions of such manifolds are discussed in
\cite{MR2547825}.
 
\begin{proposition}\label{P:semibounded}
Let $(X,\sigma)$ be a symplectically hyperbolic manifold and
let $(M,d\la)$ be its universal cover.  Then $\F G$ is a semibounded
cocycle on $\Ham(X,\sigma)$.
\end{proposition}

\begin{proof}
Let $g,h\in \Ham(X,\sigma)$ be generated by isotopies $g_t$ and $h_t$
respectively, with the corresponding Hamiltonian functions $G_t$ and
$H_t$. Let $\tilde g_t,\tilde h_t$ and $\widetilde G_t, \widetilde H_t$ be the
lifts to $M$. We need to prove that 
$$
\sup_{h\in \Ham(X,\sigma)}\F G(\tilde g,\tilde h)
$$ 
is finite.

Recall from Remark \ref{R:one-two} that 
$\F G(\tilde g,\tilde h) = 
\C K_{\la}(\tilde g)(\tilde h(x)) - \C K_{\la}(\tilde g)(x)$.  
Thus the statement will follow from the boundedness of
$\C K_{\la}(\tilde g)$ which was proven in 
\cite[Proposition 6.1]{GK1}. We recall the
proof here for the convenience of the reader.

Let $x,y\in M$ and let $C>0$ be a constant bounding the one-form
$\la$ on $M$ with respect to a Riemannian metric induced from a
metric on $X$. The first equality in the following calculation follows
from the formulae in the proof of Theorem \ref{T:ham} on page
\pageref{SS:ham} expressing $\C K_{\la}$ in terms of the action
functional.
\begin{eqnarray*}
&&|\,\C K_{\la}(\tilde g)(y) - \C K_{\la}(\tilde g)(x)\,|\\
&=&\left |\,\int_{\tilde g_t(y)}\la 
   + \int_0^1\widetilde G_t(\tilde g_t(y))dt 
   - \int_{\tilde g_t(x)}\la 
   + \int_0^1\widetilde G_t(\tilde g_t(x))dt\, \right |\\
&\leq & 2\,C\,\max_{x}\/\text{Length}(\tilde g_t(x)) +
2\,\max_{x,t} \widetilde G_t(x) < \infty
\end{eqnarray*}

The last quantity is finite because $\widetilde G_t$ and $\tilde g_t$
are lifts of $G_t$ and $g_t$ respectively and the latter
are defined on a compact manifold $X$. Also, the length is calculated
with respect to the metric induced from $X$.
We also used a straightforward fact that
$\int_{f_t(x)}\la \leq C\,\text{Length}(f_t(x))$.
\end{proof}

\begin{remark}\label{R:nondegeneracy}
The above also shows that if $|g|_\F G=0$ then $\C K_\la(g)$ is constant
and therefore $g^*\la=\la$ which cannot happen if $g\in\Ham(X,\sigma)$.
The reason why $\C K_{\la}(g)$ cannot be constant as explained in
\cite[Theorem 4.1 (1)]{GK1}.  Namely it follows from Schwarz's result
that such $g$ has two fixed points
on which the action functional (see formula \ref{Eq:action}) defining
$\C K_{\la}(g)$ assumes different values.
\end{remark}

Let $\Gamma$ be a finitely generated group.
Let $|g|_S$ denote the word length of an element $g$ of $\Gamma$ with respect
to a fixed finite set of generators $S$.

\begin{proposition}\label{P:lipshitz}
Let $\F c$ be a semibounded cocycle on $\Gamma$ then
$|\cdot|_\F c$ is Lipschitz with respect to the word-length.
More precisely
$$|g|_\F c\leq \left(2\,\max_{s\in S}|s|_\F c\right)\,|g|_S.$$
\end{proposition}

\begin{proof}
Let $s$ be one of the generators.  By Lemma \ref{L:two} we have
$$
|sg|_\F c\leq2|s|_\F c+|g|_\F c.
$$
Then, by induction,
$$
|g|_\F c=\left|s_{i_1}\ldots s_{i_{|g|_S}}\right|_\F c \leq
2\left|s_{i_1}\right|_\F c+ \ldots +2\left|s_{i_{|g|_S}}\right|_\F c
\leq 2\,
\left (\max_{s\in S}|s|_\F c \right)\,|g|_S.
$$
\end{proof}

On the other hand, as we shall explain next,
the behaviour of $\F G$ with respect to the
first argument is very different. Let
$$
\Symp(M,x,y):=\left\{f\,|\, f(x)=x \text{ and } f(y)=y\right\}
$$ 
be the subgroup consisting of symplectic diffeomorphisms preserving
the points $x,y\in M$. Let $h\in \Symp(M,d\la)$.  Define
$$
\F G_x^h\colon \Symp(M,x,h(x))\to \B R
$$
by $\F G_x^h (f) := \F G_{x,\la}(f,h)$ and observe that
it is a homomorphism of groups. 
\begin{eqnarray*}
\F G_x^h(fg) 
&=& \int_{x}^{h(x)} g^*f^*\la - \la\\
&=& \int_{x}^{h(x)} g^*f^*\la -g^*\la + g^*\la - \la\\
&=& \int_{g(x)}^{gh(x)} f^*\la - \la + \int_x^{h(x)}g^*\la - \la\\
&=& \int_{x}^{h(x)} f^*\la - \la + \int_x^{h(x)}g^*\la - \la\\
&=& \F G_x^h(f) + \F G_x^h(g) 
\end{eqnarray*}

It follows from the Stokes Lemma that
$\F G_x^h(g)$ is equal to the symplectic area of a disc
bounded by $g(\gamma) -\gamma$ where $\gamma$ is a curve
from $x$ to $h(x)$. Hence it is straightforward to show, 
by a local construction in a Darboux chart,
that if $h(x)\neq x$ then the homomorphism $\F G_x^h$ 
is nontrivial. 

\begin{proof}[Proof of Theorem \ref{T:unbounded}]
The above argument proves that the cocycle $\F G_x$ is
unbounded. Observe, e.g. by a local construction mentioned
above, that it directly applies to the subgroup
$\Ham(X,\sigma)\subset \Symp(M,d\la)$ if $(M,d\la)$ is the universal
cover of $(X,\sigma)$.
\end{proof}

\section{Applications}\label{S:applications}

\label{SS:bounded}
\subsection{Symplectic actions of finitely generated groups }

For an element $g$ of a finitely generated group $\Gamma$
one defines its {\bf translation length} as
$$
\|g\|:=\lim_{n\to\infty}{\frac{|g^n|_S}{n}},
$$
where$|g|_S$ denotes the word length of an element $g$ of $\Gamma$ with respect
to a fixed finite set of generators $S$.

\begin{remark}
Another terminology says that the cyclic subgroup generated by $g$
is {\bf undistorted} in $G$ if the translation length of $g$ does not vanish.
Observe that the (non-) vanishing of the translation length does
not depend on the choice of generators.
\end{remark}

\begin{theorem}[Polterovich {\cite[Theorem 1.6.A]{MR2003i:53126}}]
\label{T:polterovich} 
Let $(X,\sigma)$ be a closed symplectically hyperbolic manifold. If\/
$\Gamma \subset \Ham(X,\sigma)$ is a finitely generated group then
every nontrivial element of\/ $\Gamma$ has nonzero translation length.
\end{theorem}

\begin{proof}
Fix a nontrivial element $g$ in $\Gamma\subset \Ham(X,\sigma)$.  
According to a theorem of Schwarz \cite{MR1755825} (see also Theorem
9.1.6 in \cite{MR2045629}) $g$ has two contractible fixed points
$x,y\in X$ with nonzero action difference. 

Choose $h\in\Ham(X,\sigma)$ such that $h(x)=y$. Then 
$$
\F G(g,h)= \C K_{\la}(g)(h(x)) - \C K_{\la}(g)(x)\neq 0
$$ 
since this is equal to the action difference as explained in
\cite[Lemma 3.4]{GK1} and \cite[Section 2.1]{MR2003i:53126}.
Then
$$
2\,\max_{s\in S}|s|_\F G\,{\frac{|g^n|_S}{n}}\geq
\frac{|g^n|_\F G}{n}\geq
\frac{|\F G(g^n,h)|}{n}=|\F G(g,h)|,
$$
where the first inequality follows from Proposition \ref{P:semibounded}
and Proposition~\ref{P:lipshitz}, the second from the very definition
of $|\cdot|_\F G$ and the last equality from the fact that 
$\F G_x^h$ is a homomorphism. Therefore
$$
\|g\|
=\lim_{n\to\infty}\frac{|g^n|_S}{n}
\geq \frac{|\F G(g,h)|}{2\max_{s\in S}|s|_\F G}>0.
$$
\end{proof}
\begin{remark}
We gave a similar proof of this theorem in \cite{GK1}.
The new element in the above proof is the use of the
semiboundedness property of the cocycle $\F G$.
\end{remark}

\subsection{Foliated symplectic bundles}\label{SS;foliated}
Recall that the cohomology of a group $G$ is isomorphic to the
cohomology of the classifying space $BG^{d}$.
Thus the cohomology class 
$[\F G]$ is a characteristic class for symplectic 
{\em foliated bundles}. By this we mean a bundle $(M,d\la)\to E\to B$ admitting a
foliation transverse to the fibres and such that its holonomy is a
discrete subgroup of $\Symp(M,d\la)$. The corresponding characteristic
class in $H^2(B;\B R)$ will be denoted by $\F G(E)$.

We say that a bundle $L\to E'\to B$ is a {\em foliated subbundle} of
$E$ if it is a subbundle and the total space $E'$ is a union of the
leaves of the foliation in $E$. Existence of such a subbundle is
equivalent to the reduction of the structure group from
$\Symp(M,d\la)^d$ to a subgroup preserving the subspace $L\subset M$.
Here $L$ and $M$ are identified with
the fibres over $b\in B$ of $E'$ and $E$ respectively.
The following result is a direct consequence of Theorem \ref{T:vanish-isotr}.

\begin{corollary}\label{C:foliated}
Let $i\colon L\to M$ be the inclusion of an exact isotropic subma\-ni\-fold.
Let $(M,d\la)\to E\to B$ be a foliated symplectic
bundle.  If it admits a foliated subbundle $L\to E'\to B$ then the
cohomology class $\F G(E)\in H^2(B;\B R)$ is trivial.\qed
\end{corollary}

The above corollary 
gives an obstruction to the existence of foliated
subbundles with isotropic fibres. The next result, that immediately
follows from Theorem \ref{T:sas}, provides a construction of foliated
symplectic bundles with a nontrivial obstruction.

\begin{corollary}\label{C:top}
Let $(M,d\la)$ be the universal cover of a closed symplectic
$2n$-manifold $(X,\sigma)$.
The flat bundle
$$
M \to E:=M\times_{\pi_1(X)} M \to X
$$
has nontrivial characteristic class
$\F G(E)$. Moreover, the class is equal to the
cohomology class of the symplectic form and hence
$\F G(E)^n\neq 0$. \qed
\end{corollary}

\begin{example}\label{E:foliated}
\hfill
\begin{enumerate}
\item
Let $(M,d\la)\to E\to B$ is a foliated symplectic bundle
admitting a section whose image is equal to a leaf of the
foliation then $\F G(E)=~0$. Moreover, since the obstruction
$\F G(E)$ is a real cohomology class it is zero if the
bundle admits a leaf finitely covering the base.
Indeed, by pulling-back the bundle over a connected component of
such a leaf we obtain a bundle with a section. Moreover, a finite 
connected covering induces an isomorphism on the real cohomology.
\item
Let $X\to E'\to B$ be a smooth foliated bundle.
Consider vertical cotangent bundle $T^\vee X \to E \to B$.
Observe that the later is flat symplectic bundle and $E'$ is
a foliated subbundle of $E$.  Since the image of the zero section 
$X\subset T^\vee X$ is a Lagrangian submanifold we get $\F G(E)=~0$.
\item
Let $\Sigma$ be a closed and oriented surface of positive genus.
The foliated bundle 
$\widetilde \Sigma \to 
E:=\widetilde \Sigma \times_{\pi_1(\Sigma)} \widetilde \Sigma
\to \Sigma$
does not admit a foliated subbundle of positive codimension.
Indeed, it follows from Corollary~\ref{C:top} that
$\F G(E)\neq 0$. 
\item
The identity map $\Symp(M,d\la)^d \to \Symp(M,d\la)$ induces a
homomorphism $H^*(B\Symp(M,d\la);\B R) \to H^*(B\Symp(M,d\la)^d;\B R)$. 
In general, $\F G$ is not contained in the image of this homomorphism.
To see this consider $M=\B R^{2n}$. It follows from Theorem \ref{T:better-sas}
that the restriction of $\F G$ is nontrivial on 
$\B Z^{2n}\subset \Symp(\B R^{2n},d\la)$. However, the 
composition
$
\B Z^n \to \Symp(\B R^{2n},d\la)
$
factors through the contractible group $\B R^{2n}$ and hence
it induces the trivial map on cohomology.
\end{enumerate}
\hfill $\diamondsuit$
\end{example}

\bibliography{../../bib/bibliography}
\bibliographystyle{acm}

\end{document}